\documentclass[12pt]{amsart}
\usepackage{amsmath}
\usepackage{amssymb,latexsym,amsthm,amsfonts,amscd,enumerate}
\usepackage{geometry}
\usepackage{graphicx}
\usepackage[pdftex]{hyperref}
\newtheorem{theorem}{Theorem}[section]
\newtheorem{lemma}[theorem]{Lemma}
\newtheorem{proposition}[theorem]{Proposition}
\newtheorem{corollary}[theorem]{Corollary}
\newtheorem{remark}[theorem]{Remark}

\numberwithin{equation}{section}

\begin{document}

\baselineskip=15.5pt

\title[$1$-Skeleton Ideals of Graphs and Their Signless Laplace Matrices]{Standard Monomials of $1$-Skeleton Ideals of Graphs and Their Signless Laplace Matrices}

\author[C. Kumar]{Chanchal Kumar}

\address{IISER Mohali,
Knowledge City, Sector 81, SAS Nagar, Punjab -140 306, India.}

\email{chanchal@iisermohali.ac.in}

\author[G. Lather]{Gargi Lather}

\address{IISER Mohali,
Knowledge City, Sector 81, SAS Nagar, Punjab -140 306, India.}

\email{gargilather@iisermohali.ac.in}

\author[A. Roy]{Amit Roy}

\address{IISER Mohali,
Knowledge City, Sector 81, SAS Nagar, Punjab -140 306, India.}

\email{amitroy@iisermohali.ac.in}

\date{}

\begin{abstract}
Let $G$ be a (multi) graph on the vertex set $V=\{0,1,\ldots ,n\}$ with root $0$. The $G$-parking function ideal $\mathcal{M}_G$ is a monomial ideal in the polynomial ring $R=\mathbb{K}[x_1,\ldots ,x_n]$ over a field $\mathbb{K}$ such that $\dim_{\mathbb K}\left(\frac{R}{\mathcal{M}_G}\right)=\det\left(\widetilde{L}_G\right)$, where $\widetilde{L}_G$ is the truncated Laplace matrix of $G$ and $\det\left(\widetilde L_G\right)$ is the determinant of $\widetilde L_G$. In other words, standard monomials of the Artinian quotient $\frac{R}{\mathcal{M}_G}$ correspond bijectively with the spanning trees of $G$. For $0\leq k\leq n-1$, the $k$-skeleton ideal $\mathcal{M}_G^{(k)}$ of $G$ is the monomial subideal $\mathcal{M}_G^{(k)}=\left\langle m_A:\emptyset\neq A\subseteq[n]\text{ and }|A|\leq k+1\right\rangle$ of the $G$-parking function ideal
 $\mathcal{M}_G=\left\langle m_A
: \emptyset \neq A\subseteq[n]\right\rangle\subseteq R$. For  a
simple graph $G$, Dochtermann conjectured that $\dim_{\mathbb K}\left(\frac{R}{\mathcal{M}_G^{(1)}}\right)\geq\det\left(\widetilde{Q}_G\right)$, where $\widetilde Q_G$ is the truncated signless Laplace matrix of $G$. We show that Dochtermann conjecture holds for any (simple or multi) graph $G$ on $V$. 


\end{abstract}

\maketitle

\section{Introduction}
Let $G$ be a multigraph on the vertex set $V=\{0,1,\ldots ,n\}=\{0\}\cup[n]$ with root $0$ and adjacency matrix $A(G)=[a_{ij}]_{0\leq i,j\leq n}$. Let $E(i,j)$ be the set of edges between $i,j\in V$. Then $E(i,j)=E(j,i)$ and $|E(i,j)|=a_{ij}=a_{ji}$. We always assume that $G$ is loopless, i.e., $a_{ii}=0$ for every $ i\in V$. For $\emptyset\neq A\subseteq [n]=\{1,2,\ldots ,n\}$, set $d_A(i)=\sum_{j\in V\setminus A}a_{ij}$, for $i\in A$. Then $d_i=d_{\{i\}}(i)$ is the degree of the vertex $i$ in $G$. Let $D={\rm diag}[d_0,d_1,\ldots ,d_n]$ be the diagonal matrix of order $n+1$. The Laplace matrix $L_G$ and the signless Laplace matrix $Q_G$ of $G$ are given by
\[L_G=D-A(G) \quad {\rm and}\quad Q_G=D+A(G).\] On deleting row and column corresponding to the root $0$ from $L_G$ and $Q_G$, we  obtain truncated Laplace matrix $\widetilde L_G$ and truncated signless Laplace matrix $\widetilde Q_G$ of $G$, respectively. Let $R=\mathbb{K}[x_1,\ldots ,x_n]$ be the polynomial ring in $x_1,\ldots ,x_n$ over a field $\mathbb{K}$. Sometimes, we write $R=R_n$ to indicate the number of variables in the polynomial ring. The monomial ideal $\mathcal{M}_G$ in $R$ given by 
\[\mathcal{M}_G=\left\langle m_A=\prod_{i\in A}x_i^{d_A(i)}:\emptyset\neq A\subseteq [n]\right\rangle\]
is called the {\em $G$-parking function ideal}. The monomial ideal $\mathcal{M}_G$, more generally for directed graph $G$ on $V$, has been introduced by Postnikov and Shapiro \cite{PS}. The standard monomials $\mathbf{x}^{\mathbf{p}}=x_1^{p_1}x_2^{p_2}\cdots x_n^{p_n}$ of $\frac{R}{\mathcal{M}_G}$ (i.e., $\mathbf{x}^{\mathbf{p}}\notin \mathcal{M}_G$) correspond to $G$-parking functions $\mathbf{p}=(p_1,\ldots,p_n)
\in \mathbb{N}^n$ and $\dim_{\mathbb K}\left(\frac{R}{\mathcal{M}_G}\right)=\det\widetilde L_G$. Thus by the matrix tree theorem, number of $G$-parking functions equals the number of spanning trees of $G$. An algorithmic bijection between the set of $G$-parking functions and the set of spanning trees of $G$ is given by Perkinson, Yang and Yu \cite{PYY} for simple graph and by Gaydarov and Hopkins \cite{GH} for multigraph.

For $0\leq k\leq n-1$, we consider the monomial subideal $\mathcal{M}_G^{(k)}$ of $\mathcal{M}_G$ given by $\mathcal{M}_G^{(k)}=\left\langle m_A=\prod_{i\in A}x_i^{d_A(i)}:\emptyset\neq A\subseteq [n] \quad {\rm and}\quad |A|\leq k+1\right\rangle$ and call it the {\em $k$-skeleton ideal} of $G$. Dochtermann \cite{D,Do} showed that like $\mathcal{M}_G$, the $k$-skeleton ideals $\mathcal{M}_G^{(k)}$ of $G$ also have many interesting combinatorial properties. He 
verified that 
$\dim_{\mathbb K}\left(\frac{R}{\mathcal{M}_G^{(1)}}\right)=\det\widetilde Q_G$ for the complete graph $G=K_{n+1}$ and conjectured the inequality 
$\dim_{\mathbb K}\left(\frac{R}{\mathcal{M}_G^{(1)}}\right)\geq\det\widetilde Q_G$ for any simple graph $G$ on $V$.

Let $a,b$ be positive integers. The complete multigraph $K_{n+1}^{a,b}$ on $V$ is given by the adjacency matrix $A\left(K_{n+1}^{a,b}\right)=[a_{ij}]_{0\leq i,j\leq n}$ with $a_{i0}=a_{0i}=a$ and $a_{ij}=a_{ji}=b$ for $i,j\in[n]$; $i\neq j$. Let $G$ be a subgraph of the complete multigraph $K_{n+1}^{a,b}$ obtained by deleting some edges through the root $0$. Then we show (Theorem \ref{RC}) that $\dim_{\mathbb{K}}\left(\frac{R}{\mathcal{M}_G^{(1)}}\right)=\det\widetilde Q_G$. Also for any (simple or multi) graph $G$ on $V$, we show (Corollary \ref{MG}) that $\dim_{\mathbb{K}}\left(\frac{R}{\mathcal{M}_G^{(1)}}\right)\geq\det\widetilde Q_G$. In fact, corresponding to any positive semidefinite matrix $H=[\alpha_{ij}]_{n\times n}$ over $\mathbb{N}$ satisfying $\alpha_i=\alpha_{ii}
\geq \max_{j \neq i} \alpha_{ij}$ for $1 \le i \le n$, we consider
the  monomial ideal
\[\mathcal{J}_H=\left\langle x_l^{\alpha_l},x_i^{\alpha_i-\alpha_{ij}}x_j^{\alpha_j-\alpha_{ij}}:1\leq l\leq n,1\leq i < j\leq n\right\rangle\]
in $R$. Using Courant-Weyl inequalities and Fischer's inequality on the determinant of
positive semidefinite matrices, we obtain (Theorem \ref{MT}) $\dim_{\mathbb{K}}\left(\frac{R}{\mathcal{J}_H}\right)\geq\det H$.

\section{Complete multigraphs and parking functions}

Let $K_{n+1}$ be the complete (simple) graph on $V$. Then $K_{n+1}$-parking functions are precisely (ordinary) parking functions of length $n$. More generally, if $\lambda=(\lambda_1,\lambda_2,\ldots ,\lambda_n)\in\mathbb{N}^n$ with $\lambda_1\geq\lambda_2\geq\cdots\geq\lambda_n\geq 1$, then a finite sequence $\mathbf{p}=(p_1,\ldots ,p_n)\in\mathbb{N}^n$ is called a {\em $\lambda$-parking function} if a non-decreasing rearrangement $p_{j_1}\leq p_{j_2}\leq\cdots\leq p_{j_n}$ of $\mathbf{p}$ satisfies $p_{j_i}<\lambda_{n-i+1}$ for $1 \le i \le n$. Let ${\rm PF}(\lambda)$ be the set of $\lambda$-parking functions. An ordinary parking function of length $n$ is a $\lambda$-parking function for $\lambda=(n,n-1,\ldots ,2,1)$.

For $\lambda=(\lambda_1,\ldots ,\lambda_n)\in\mathbb{N}^n$ with $\lambda_1\geq\lambda_2\geq\cdots\geq\lambda_n\geq 1$, let
\[\Lambda(\lambda)=\Lambda(\lambda_1,\lambda_2,\ldots ,\lambda_n)=
\left[
\frac{\lambda_{n-i+1}^{j-i+1}}{(j-i+1)!}
\right]_{1\leq i,j\leq n}
\]
be a $n\times n$ Steck matrix, whose $(i,j)$th entry is $\frac{\lambda_{n-i+1}^{j-i+1}}{(j-i+1)!}$ if $i\leq j+1$, and $0$, otherwise. Consider the monomial ideal $\mathcal{M}_{\lambda}=\left\langle \left(\prod_{i\in A}x_i\right)^{\lambda_{|A|}}:\emptyset\neq A\subseteq[n]\right\rangle$ in $R$. Then the standard monomials of $\frac{R}{\mathcal{M}_{\lambda}}$ are precisely $\lambda$-parking functions and by Steck determinant formula (see \cite{PiSt}), the number of $\lambda$-parking function is given by
\[\dim_{\mathbb{K}}\left(\frac{R}{\mathcal{M}_{\lambda}}\right)=|{\rm PF}(\lambda)|=
n!\det\left(\Lambda(\lambda)\right).\]
The Steck determinant $\det(\Lambda(\lambda)) $ can be easily evaluated for the sequence 
$\lambda = (\lambda_1,\ldots,\lambda_n)$ of special types, for example, it is in arithmetic progression
(see \cite{CK, KY, PiSt}). Let $x$ be a variable and $b\in\mathbb N$. Suppose $f_n^b(x)=\det(\Lambda(x+(n-1)b,x+(n-2)b,\ldots ,x+b,x))$ and $g_n^b(x)=\det\left(\Lambda(x+b,\underbrace{x,\ldots ,x}_{n-1})\right)$. Then $f_n^b(x)$ and $g_n^b(x)$ are polynomials in $x$ of degree $n$ given by
\[f_n^b(x)=\frac{x(x+nb)^{n-1}}{n!} \quad {\rm and}
\quad g_n^b(x)=\frac{x^{n-1}(x+nb)}{n!}.\]
In case $b=1$, we get $f_n(x)=f_n^1(x)=\frac{x(x+n)^{n-1}}{n!}$  and $ g_n(x)=g_n^1(x) = 
\frac{x^{n-1}(x+n)}{n!}$. Hence for the complete graph $K_{n+1}$, we have
\[\dim_{\mathbb K}\left(\frac{R}{\mathcal{M}_{K_{n+1}}}\right)=(n+1)^{n-1} \quad {\rm and}\quad \dim_{\mathbb K}\left(\frac{R}{\mathcal{M}_{K_{n+1}}^{(1)}}\right)=(n-1)^{n-1}(2n-1)=\det\left(\widetilde Q_{K_{n+1}}\right).\]
More generally, for complete multigraph $K_{n+1}^{a,b}$,
we have
\[\dim_{\mathbb K}\left(\frac{R}{\mathcal{M}_{K_{n+1}^{a,b}}}\right)=a(a+nb)^{n-1} \quad {\rm and}\quad \dim_{\mathbb K}\left(\frac{R}{\mathcal{M}_{K_{n+1}^{a,b}}^{(1)}}\right)= (n!) g_n^b(a+(n-2)b) .\]
It can be easily verfied that
\begin{eqnarray}
\dim_{\mathbb K}\left(\frac{R}{\mathcal{M}_{K_{n+1}^{a,b}}^{(1)}}\right)=(a+(n-2)b)^{n-1}(a+(2n-2)b)
=\det\left(\widetilde Q_{K_{n+1}^{a,b}}\right).
\label{e0}
\end{eqnarray}

Let $0\leq r\leq n$ and $G_{n,r}$ be the graph obtained from $K_{n+1}$ on deleting precisely $r$ edges through root $0$. We have $G_{n,0}=K_{n+1}$. On renumbering vertices, we assume that the deleted edges are between $0$ and $i$ for $n-r+1\leq i\leq n$. We proceed to verify that 
$\dim_{\mathbb K}\left(\frac{R}{\mathcal{M}_{G_{n,r}}^{(1)}}\right)=\det\left(\widetilde Q_{G_{n,r}}\right)$.

Let $a$ be a fixed positive integer and let $\omega$ be a weight function (depending on $r\in[0,n]$) given by 
$\omega(i)=
\begin{cases}
a \quad  &{\rm if}\quad i\in[n-r], \\
a-1 \quad &{\rm if}\quad i\in[n]\setminus [n-r].
\end{cases}$
Let $\mathcal{I}_{n,r}^{\langle a\rangle}$ be a monomial ideal in $R_n=\mathbb K[x_1,\ldots ,x_n]$ given by 
\[ \mathcal{I}_{n,r}^{\langle a\rangle}=\left\langle x_i^{\omega(i)},x_i^{\omega(i)-1}x_j^{\omega(j)-1}:i,j\in[n] \quad {\rm and} \quad i\neq j\right\rangle.
\]
Clearly, $\mathcal{I}_{n,r}^{\langle n \rangle}=
\mathcal{M}_{G_{n,r}}^{(1)}.$ 

Consider the map $\mu = \mu_{x_{n-r+1}}:R_n\rightarrow \frac{R_n}{\mathcal{I}_{n,r-1}^{\langle a\rangle}}$ given by $\mu(f)=x_{n-r+1}f + \mathcal{I}_{n,r-1}^{\langle a\rangle}$ for $f\in R_n$. Then $\ker\mu =\left(\mathcal{I}_{n,r-1}^{\langle a\rangle} :
x_{n-r+1}\right)$ and let $\bar{\mu} : \frac{R_n}{
\left(\mathcal{I}_{n,r-1}^{\langle a\rangle} ~:~
x_{n-r+1}
\right)} \rightarrow 
\frac{R_n}{\mathcal{I}_{n,r-1}^{\langle a\rangle}}$
be the induced $R_n$-linear map. Thus 
there exists a short exact sequence of 
$R_n$ modules (or $\mathbb K$-vector spaces)
\begin{align}\label{KL}
0\rightarrow \frac{R_n}{\left(\mathcal{I}_{n,r-1}^{\langle a\rangle} ~:~
x_{n-r+1}
\right)}\xrightarrow{\bar{\mu}}\frac{R_n}{\mathcal{I}_{n,r-1}^{\langle a\rangle}}\xrightarrow{\nu}\frac{R_n}{\langle \mathcal{I}_{n,r-1}^{\langle a\rangle},x_{n-r+1}\rangle}\rightarrow 0,
\end{align}
where $\nu$ is the natural projection.

\begin{lemma}
Let $r\geq 1$. Then 
\begin{enumerate}
\item[{\rm (i)}] $\left(\mathcal{I}_{n,r-1}^{\langle a\rangle}:x_{n-r+1}\right)=\mathcal{I}_{n,r}^{\langle a\rangle}$ .
\item[{\rm (ii)}] $\dim_{\mathbb K}\left(\frac{R_n}{\mathcal{I}_{n,r}^{\langle a\rangle}}\right)=\dim_{\mathbb K}\left(\frac{R_n}{\mathcal{I}_{n,r-1}^{\langle a\rangle}}\right)-\dim_{\mathbb K}\left(\frac{R_{n-1}}{\mathcal{I}_{n-1,r-1}^{\langle a\rangle}}\right)$.
\end{enumerate}
\label{Lemma1}
\end{lemma}
\begin{proof}
Clearly,
$\left(\mathcal{I}_{n,r-1}^{\langle a\rangle}
 : x_{n-r+1}\right)=
 \{ f \in R_n : x_{n-r+1} f \in 
 \mathcal{I}_{n,r-1}^{\langle a\rangle} \} =
 \mathcal{I}_{n,r}^{\langle a\rangle}$.
 Thus the short exact sequence (\ref{KL}) is
\(
0\rightarrow \frac{R_n}{\mathcal{I}_{n,r}^{\langle a\rangle}}\xrightarrow{\bar{\mu}}\frac{R_n}{\mathcal{I}_{n,r-1}^{\langle a\rangle}}\xrightarrow{\nu}\frac{R_n}{\langle \mathcal{I}_{n,r-1}^{\langle a\rangle},x_{n-r+1}\rangle}\rightarrow 0. 
\)
Further, We see that $\left\langle \mathcal{I}_{n,r-1}^{\langle a\rangle},~ x_{n-r+1}
\right\rangle = 
\left\langle \mathcal{I}_{n,r}^{\langle a\rangle},~ 
x_{n-r+1} \right\rangle$. Also, 
$\frac{R_n}{\left\langle 
\mathcal{I}_{n,r}^{\langle a\rangle},~ 
x_{n-r+1}\right\rangle} \cong 
\frac{R_{n-1}}
{\mathcal{I}_{n-1,r-1}^{\langle a\rangle}}$ as 
$\mathbb K$-vector spaces. Thus from the
short exact sequence of $\mathbb K$ vector spaces, we have 
\[\dim_{\mathbb K}\left(\frac{R_n}{\mathcal{I}_{n,r}^{\langle a\rangle}}\right)=\dim_{\mathbb K}\left(\frac{R_n}{\mathcal{I}_{n,r-1}^{\langle a\rangle}}\right)-\dim_{\mathbb K}\left(\frac{R_{n}}{\left\langle \mathcal{I}_{n,r-1}^{\langle a\rangle},
~ x_{n-r+1}\right\rangle}\right).\]
\hfill $\square$
\end{proof}

\begin{lemma}\label{Lemma2} Let $0 \le r \le n$. Then
\begin{enumerate}
\item[{\rm (i)}] $\dim_{\mathbb K}\left(\frac{R_n}{\mathcal{I}_{n,0}^{\langle a\rangle}}\right)=(a-1)^{n-1}(a+(n-1))$.
\item[{\rm (ii)}]
\(\dim_{\mathbb K}\left(\frac{R_n}{\mathcal{I}_{n,r}^{\langle a\rangle}}\right)=\sum_{i=0}^r(-1)^i\binom{r}{i}(a-1)^{n-i-1}(a+(n-i-1))=\sum_{i=0}^r(-1)^i\binom{r}{i}\theta_{n-i}(a-1),\)
where $\theta_l(x)=x^{l-1}(x+l)$ is a polynomial in $x$.
\end{enumerate}
\end{lemma}
\begin{proof}
We have $\mathcal{I}_{n,0}^{\langle a\rangle}=\left\langle x_i^a, (x_ix_j)^{a-1} : i,j \in [n]; i\neq j
\right\rangle$. Thus $\dim_{\mathbb K}\left(\frac{R_n}
{\mathcal{I}_{n,0}^{\langle a\rangle}}\right)=$ number of $\lambda$-parking functions for $\lambda=(a,a-1,\ldots ,a-1)\in\mathbb N^n$. Here $\lambda=(x+1,x,\ldots ,x)$ for $x=a-1$ and $b=1$. Therefore,
\begin{align*}
\dim_{\mathbb K}\left(\frac{R_n}{\mathcal{I}_{n,0}^{\langle a\rangle}}\right)=n!\det\left(\Lambda(\lambda)\right)=(n!)g_n(a-1) 
=(a-1)^{n-1}(a+n-1).
\end{align*}
This proves {\rm (i)}.
We shall prove {\rm (ii)} by induction on $r$. For $r=0$, it follows from {\rm (i)}.
Assume $r\geq 1$. From Lemma \ref{Lemma1}, we have
\[
\dim_{\mathbb K}\left(\frac{R_n}{\mathcal{I}_{n,r}^{\langle a\rangle}}\right)=\dim_{\mathbb K}\left(\frac{R_n}{\mathcal{I}_{n,r-1}^{\langle a\rangle}}\right)-\dim_{\mathbb K}\left(\frac{R_{n-1}}{\mathcal{I}_{n-1,r-1}^{\langle a\rangle}}\right).
\]
For $n\geq r\geq 1$,
by induction assumption, 
$\dim_{\mathbb K}\left(\frac{R_n}{\mathcal{I}_{n,r-1}^{\langle a\rangle}}\right)=\sum_{i=0}^{r-1}(-1)^i\binom{r-1}{i}\theta_{n-i}(a-1)$ and 
\(
\dim_{\mathbb K}\left(\frac{R_{n-1}}{\mathcal{I}_{n-1,r-1}^{\langle a\rangle}}\right)=\sum_{i=0}^{r-1}(-1)^i\binom{r-1}{i}\theta_{n-1-i}(a-1) 
=\sum_{i=1}^{r}(-1)^{i-1}\binom{r-1}{i-1}\theta_{n-i}(a-1)\).
Thus 
\begin{align*}
\dim_{\mathbb K}\left(\frac{R_n}{\mathcal{I}_{n,r}^{\langle a\rangle}}\right)&=\sum_{i=0}^{r-1}(-1)^i\binom{r-1}{i}\theta_{n-i}(a-1)+\sum_{i=1}^{r}(-1)^i\binom{r-1}{i-1}\theta_{n-i}(a-1) \\
&=\theta_n(a-1)+\sum_{i=1}^{r-1}(-1)^i\left[\binom{r-1}{i}+\binom{r-1}{i-1}\right]\theta_{n-i}(a-1)+(-1)^r\theta_{n-r}(a-1) \\
&=\sum_{i=0}^{r}(-1)^i\binom{r}{i}\theta_{n-i}(a-1).
\end{align*}
\hfill $\square$
\end{proof}
\begin{remark} {\rm 
 Note that $I_{n,n}^{\langle a\rangle}=
 \mathcal{I}_{n,0}^{\langle a-1\rangle}$ for 
 $a \geq 2$. Thus from Lemma \ref{Lemma2}, we obtain an 
 interesting combinatorial identity :
\[
(a-2)^{n-1}(a+(n-2)) = \sum_{i=0}^n(-1)^i \binom{n}{i}(a-1)^{n-i-1}(a+(n-i-1)) \quad {\rm for} \quad n\geq 0.
\]
Being a polynomial identity in $a$, it is valid for any $a\in\mathbb R$. }
\end{remark}
\begin{proposition}\label{p1}
$\dim_{\mathbb K}\left(\frac{R_n}{\mathcal{M}_{G_{n,r}}^{(1)}}\right)=\det\left(\widetilde Q_{G_{n,r}}\right)$.
\end{proposition}
\begin{proof}
The determinant of the truncated signless Laplace matrix $\widetilde Q_{G_{n,r}}$ of $G_{n,r}$ is given by
\begin{align}\label{e1}
\det\left(\widetilde Q_{G_{n,r}}\right) 
= (n-1)^{n-r-1}(n-2)^{r-1}\left[(2n-1)(n-2)+r\right].
\end{align}
In fact, on applying the column operation $C_1+(C_2+\ldots +C_n)$ on $\widetilde Q_{G_{n,r}}$,
 followed by the row operations $R_2-R_1,R_3-R_1,\ldots ,R_n-R_1$, $\widetilde Q_{G_{n,r}}$ reduces to the matrix 
\[ \begin{bmatrix} 
2n-1 & 1 & \cdots & 1 & 1 & \cdots & 1 \\
0 & n-1 & \cdots & 0 & 0 & \cdots & 0 \\
\vdots & \vdots & \ddots & \vdots & \vdots & \ddots & \vdots \\
0 & 0 & \cdots & n-1 & 0 & \cdots & 0 \\
-1 & 0 & \cdots & 0 & n-2 & \cdots & 0 \\
\vdots & \vdots & \ddots & \vdots & \vdots & \ddots & \vdots \\
-1 & 0 & \cdots & 0 & 0 & \cdots & n-2
\end{bmatrix}_ {n \times n} ,
\]
where $n-2$ appears as the diagonal entry
in the last $r$ rows.
Now expanding the determinant along the first column, we get (\ref{e1}).

Also, $\mathcal{I}_{n,r}^{\langle n\rangle} =
\mathcal{M}_{G_{n,r}}^{(1)}$ and from 
Lemma \ref{Lemma2}, we have 
\begin{eqnarray*}
\dim_{{\mathbb K}} \left(
\frac{R_n}{\mathcal{M}_{G_{n,r}}^{(1)}} \right)
& = & \sum_{i = 0}^{r} ~(-1)^i
\binom{r}{i} (n-1)^{n-i-1}((2n-1)-i) \\
& = & (n-1)^{n-r-1}  (2n-1) 
\left\lbrace \sum_{i=0}^r ~(-1)^i 
\binom{r}{i}(n-1)^{r-i} \right\rbrace  \\
 & & \quad + ~(n-1)^{n-r-1}
 \left\lbrace \sum_{i=0}^r(-1)^{i+1}
\binom{r}{i}~ i (n-1)^{r-i} \right\rbrace \\
& = & (n-1)^{n-r-1} \left[ (2n-1)(n-2)^r
+r(n-2)^{r-1} \right] \\
& = &(n-1)^{n-r-1}(n-2)^{r-1} \left[(2n-1)(n-2)
+r \right] \\
& = & \det \left( \widetilde{Q}_{G_{n,r}} \right).
\end{eqnarray*}
\hfill $\square$
\end{proof}

We now proceed to generalize Proposition \ref{p1} to multigraphs.
\begin{theorem}\label{RC}
Let $G$ be a multigraph on $V$ obtained from the complete multigraph $K_{n+1}^{a,b}$ on deleting some edges through the root $0$. Then
\begin{align}\label{e2}
\dim_{\mathbb K}\left(\frac{R_n}{\mathcal{M}_G^{(1)}}\right)=\det\widetilde Q_G .
\end{align}
\end{theorem}
\begin{proof} We shall prove this theorem
by induction on $n$. For $n=1$, $G=K_2^{a,0}$ for some $a\geq 0$. Then $\mathcal{M}_G^{(1)}=\langle x_1^a\rangle\subseteq R_1$ and $\widetilde Q_G=[a]$ and hence (\ref{e2}) holds. 
For $n=2$, the adjacency matrix $A(G)=
\begin{bmatrix}
0 & a_1 & a_2 \\
a_1 & 0 & b \\
a_2 & b & 0
\end{bmatrix}_{3\times 3}
$ for some $a_1,a_2\leq a$ and $b\geq 1$. Then $\mathcal{M}_G^{(1)}=\left\langle x_1^{a_1+b},x_2^{a_2+b},x_1^{a_1}x_2^{a_2}\right\rangle$ and $\widetilde Q_G=
\begin{bmatrix}
a_1+b & b \\
b & a_2+b
\end{bmatrix}_{2\times 2}.
$ Again, $\dim_{\mathbb K}\left(\frac{R_2}{\mathcal{M}_G^{(1)}}\right)=(a_1+b)(a_2+b)-b^2=\det\widetilde Q_G$ shows that (\ref{e2}) holds.
By induction assumption, suppose theorem holds for multigraphs on the vertex set $\{0,1,\ldots ,m\}$; $m<n$, obtained from $K_{m+1}^{a,b}$ on deleting some edges through the root $0$ for any $a,b\geq 1$. 

Let $n\geq 3$ and $G$ be a multigraph on $V=\{0,1,\ldots ,n\}$ obtained from $K_{n+1}^{a,b}$ on deleting some edges through the root $0$. The adjacency matrix $A(G)=[a_{ij}]_{(n+1)\times(n+1)}$ of $G$ satisfies
 $a_{0i}=a_{i0}=a_i\leq a$  and $a_{ij}=b$ for $i,j\in[n]$ with $i\neq j$. Then 
\[
\mathcal{M}_G^{(1)}=\left\langle x_l^{a_l+(n-1)b},x_i^{a_i+(n-2)b}x_j^{a_j+(n-2)b}:1\leq l\leq n,1\leq i<j\leq n\right\rangle .
\]
Let $e_0$ be a fixed edge from $0$ to $j$ in $G$ $(1\leq j\leq n)$. Consider the multigraph $G_1=G-e_0$ obtained from $G$ on deleting the edge $e_0$. Then clearly,
\[
\mathcal{M}_{G_1}^{(1)} = \left(\mathcal{M}_G^{(1)}:x_j\right)=\left\{f\in R_n:x_jf\in \mathcal{M}_G^{(1)}\right\}.
\]
Consider the
$R_n$-linear map 
$ \mu_{x_j} : R_n \rightarrow \frac{R_n}
{\mathcal{M}_G^{(1)}}$ given by
$\mu_{x_j}(f) = x_j f + \mathcal{M}_G^{(1)}$ for $f \in R_n$. Then ${\rm Ker}(\mu_{x_j}) =
\left(\mathcal{M}_G^{(1)}:x_j\right) = 
\mathcal{M}_{G_1}^{(1)}$ and there
is a short exact 
sequence of $\mathbb K$-vector spaces 
\begin{align}\label{e4}
0\rightarrow\frac{R_n}{\mathcal{M}_{G_1}^{(1)}}\xrightarrow{\bar{\mu}_{x_j}}\frac{R_n}{\mathcal{M}_G^{(1)}}\xrightarrow{\nu}\frac{R_n}{\left\langle \mathcal{M}_G^{(1)},x_j\right\rangle}\rightarrow 0 ,
\end{align}
where $\nu$ is the 
natural projection and
$\bar{\mu}_{x_j}$ is the map induced by 
$ \mu_{x_j}$.
 Let $G_2$ be a multigraph on the vertex set $V\setminus\{j\}$ with adjacency matrix $A(G_2)=\left[a_{rs}^{(2)}\right]_{\underset{r,s\neq j}{0\leq r,s\leq n}}$, where $a_{0,r}^{(2)}=a_r+b$, $a_{rs}^{(2)}=b$ for $r,s\in [n]\setminus\{j\}$, $r\neq s$.
Then, writing $R_{n-1}=\mathbb K[x_1,\ldots ,\hat{x}_j,\ldots ,x_n]$ for the polynomial ring over $\mathbb K$ in $n-1$ variables $x_1,\ldots ,x_{j-1},x_{j+1},\ldots ,x_n$, we have 
\(
\frac{R_{n-1}}{\mathcal{M}_{G_2}^{(1)}}\cong\frac{R_n}{\left\langle \mathcal{M}_G^{(1)},x_j\right\rangle}\).

Thus from the short exact sequence (\ref{e4}), we get
\begin{align}\label{e5}
\dim_{\mathbb K}\left(\frac{R_n}{\mathcal{M}_G^{(1)}}\right)=\dim_{\mathbb K}\left(\frac{R_n}{\mathcal{M}_{G_1}^{(1)}}\right)+\dim_{\mathbb K}\left(\frac{R_{n-1}}{\mathcal{M}_{G_2}^{(1)}}\right).
\end{align}
As determinant is linear on columns, we have
\begin{align}\label{e31}
\det\left(\widetilde Q_G\right)=\det\left(\widetilde Q_{G_1}\right)+\det\left(\widetilde Q_{G_2}\right).
\end{align}
By induction assumption, $\dim_{\mathbb K}\left(\frac{R_{n-1}}{\mathcal{M}_{G_2}^{(1)}}\right)=\det\left(\widetilde{Q}_{G_2}\right)$.
Thus from (\ref{e5}) and (\ref{e31}), we see that
\begin{align*}
\dim_{\mathbb K}\left(\frac{R_n}{\mathcal{M}_G^{(1)}}\right)=\det\left(\widetilde Q_G\right)\iff \dim_{\mathbb K}\left(\frac{R_n}{\mathcal{M}_{G_1}^{(1)}}\right)=\det\left(\widetilde Q_{G_1}\right).
\end{align*}
In other words, if theorem holds for a multigraph $G$ on $V$ then it also holds for the multigraph $G_1=G\setminus e_0$, and vice-versa. From (\ref{e0}),
\(
\dim_{\mathbb K}\left(\frac{R_n}{\mathcal{M}_{K_{n+1}^{a,b}}^{(1)}}\right)=\det\left(\widetilde Q_{K_{n+1}^{a,b}}\right).
\)
Thus, we see that the theorem holds for $G$ by deleting edges through the root, one by one.
\hfill $\square$
\end{proof}

\section{Positive semidefinite matrices over Nonnegative
Integers}
Let $n\geq 1$ and $M_n(\mathbb{N})$ be the set of
$n \times n$ matrices over nonnegative integers
$\mathbb{N}$. 
Let $\mathcal{G}_n=\{H=[b_{ij}]\in M_n(\mathbb N):H^t=H ~{\rm and} ~b_{ii}
\geq \max_{j \neq i} b_{ij} 
~{\rm for} ~1 \le i \le n \}$, where $H^t$ =
transpose of $H$.
 For 
$H=[b_{ij}]_{n\times n}\in\mathcal G_n$ with $\alpha_i=b_{ii} $, we consider the monomial ideal 
$\mathcal{J}_H = 
\left\langle x_i^{\alpha_i},
x_i^{\alpha_i-b_{ij}}x_j^{\alpha_j-b_{ij}}
: i,j \in [n];~i \ne j \right\rangle$ in the polynomial ring $R_n=\mathbb K[x_1,\ldots ,x_n]$.
If $H=\widetilde Q_G$, the truncated signless Laplace matrix of a multigraph $G$ on $V$, then 
$\mathcal{J}_H
= \mathcal{M}_G^{(1)}$. We shall show that $\dim_{\mathbb K}\left(\frac{R_n}
{\mathcal{J}_H}\right) \geq \det H$ for every positive semidefinite $H \in \mathcal{G}_n$. 
For this, we need the following results on symmetric or Hermitian matrices.

Let $A\in M_n(\mathbb C)$ be a Hermitian matrix and its real eigenvalues be arranged in a non-decreasing order $\lambda_1(A)\leq\lambda_2(A)\leq\cdots\leq\lambda_n(A)$. The Courant-Weyl inequalities (see \cite{HJ}) compare eigenvalues of two Hermitian matrices with their sum.
\begin{theorem}[Courant-Weyl]
Let $A,B\in M_n({\mathbb C})$ be Hermitian matrices. Then
\[
\lambda_i(A+B)
~\leq ~ \lambda_{i+j}(A)+\lambda_{n-j}(B)\quad
 for~ j=0,1,\ldots ,n-i.
\]
\label{CW}
\end{theorem}
Hadamard showed that the determinant of a positive definite matrix $M=[\alpha_{ij}]_{n\times n}$ is bounded by the product of its diagonal entries, i.e., $\det(M)\leq\alpha_{11}\alpha_{22}\cdots\alpha_{nn}$. Fischer's inequality (see \cite{HJ}) is a generalization of Hadamard's theorem.
\begin{theorem}[Fischer]
Let $M\in M_n(\mathbb C)$ be a positive semidefinite matrix having block decomposition $M=\begin{bmatrix}
A & B\\
B^{*} & C
\end{bmatrix} $
with square matrices $A$ and $C$. Then
\(
\det M ~\leq ~\det(A)\det(C).
\)
\label{F}
\end{theorem}

For a proof of Theorem \ref{CW} and Theorem \ref{F}, we
refer to the book of Horn and Johnson \cite{HJ}. Now,
using Courant-Weyl inequalities and Fischer's inequality, we prove the following result.
\begin{theorem}\label{MT}
Let $H\in\mathcal G_n$ be  positive semidefinite
 and $\mathcal{J}_H$ be the monomial ideal in the polynomial ring $R=R_n$ associated to $H$. Then 
\[
\dim_{\mathbb K}\left(\frac{R_n}{\mathcal{J}_H}\right)\geq\det H.
\]
\end{theorem}
\begin{proof}
We shall proof this theorem by induction on the order $n$ of $H$. For $n=1$, $H=[\alpha_1]_{1\times 1}$ and 
$\mathcal{J}_H =
\left\langle x_1^{\alpha_1} \right\rangle$, 
and thus $\dim_{\mathbb K}\left(\frac{R_1}
{\mathcal{J}_H}\right) = \alpha_1 = \det H$. 
For $n=2$, $H=
\begin{bmatrix}
\alpha_1 & a_{12} \\
a_{12} & \alpha_2
\end{bmatrix}_{2\times 2}
$ and $\mathcal{J}_H = \left\langle x_1^{\alpha_1},
x_2^{\alpha_2}, x_1^{\alpha_1-a_{12}}
x_2^{\alpha_2-a_{12}} \right\rangle \subseteq R_2$.
Again, $\dim_{\mathbb K}\left(\frac{R_2}
{\mathcal{J}_H}\right) = \alpha_1\alpha_2-a_{12}^2
= \det H$. Assume that $n\geq 3$ and the theorem holds for every positive semidefinite matrices in 
$\mathcal{G}_m$ for $1 \le m < n$. 
Let $H=[a_{ij}]_{n\times n} \in \mathcal{G}_n$ with
$\alpha_i=a_{ii}$. 
Let $b= \max \{ a_{ij} : i,j \in [n]; ~i \ne j\}$.
 On permuting rows and columns of $H$, obtain $H'=\left[a_{ij}'\right] \in \mathcal{G}_n$ similar to $H$ such that there exists an integer $r$ ($0
 \leq r \leq n-2$) satisfying $a_{i,r+1}' < b$ 
 and $a_{r+1,j}' = b$ for $1 \leq i < r+1 < j \leq n$. The monomial ideal $\mathcal{J}_{H'}$ 
 is obtained from $\mathcal{J}_H$ by renumbering variables. Thus
  $\dim_{\mathbb K}\left(\frac{R_n}{\mathcal{J}_H}\right) = \dim_{\mathbb K}\left(\frac{R_n}
  {\mathcal{J}_{H'}}\right)$ and 
  $\det H = \det H'$. Hence, 
  without loss of generality, assume 
  that $H = H'$, i.e., 
  there exists $r~(0 \leq r \leq n-2)$ such 
  that $a_{i,r+1} < b$ and $a_{r+1,j} = b$ 
  for $1 \leq i< r+1 < j \leq n$. 
  Let $\mu^{\prime} : R_n \rightarrow \frac{R_n}{\mathcal{J}_H}$
  be the $R_n$-linear map given by 
  $\mu^{\prime}(f) = 
  x_{r+1}^{\alpha_{r+1} -b} f 
  + \mathcal{J}_H$ for $ f \in R_n$.
  Then $ \ker \mu^{\prime} = \left(\mathcal{J}_H : 
  x_{r+1}^{\alpha_{r+1}-b} \right) = 
  \left\{f \in R_n : x_{r+1}^{\alpha_{r+1}-b}f \in 
  \mathcal{J}_H \right\}$.
  Now as in (\ref{KL}), there is a 
  short exact sequence of 
  $\mathbb K$-vector spaces,
\begin{align}\label{e6}
0 \rightarrow \frac{R_n}{\left(\mathcal{J}_H : 
x_{r+1}^{\alpha_{r+1}-b}\right)}
\xrightarrow{{\bar{\mu}}^{\prime}} \frac{R_n}{\mathcal{J}_H}\xrightarrow{\nu} \frac{R_n}
{\left\langle \mathcal{J}_H,~
x_{r+1}^{\alpha_{r+1}-b} \right\rangle}
\rightarrow 0,
\end{align}
where $\nu$ is natural projection and
${\bar{\mu}}^{\prime}$ is the map induced by 
$\mu^{\prime}$.

Let $H_1=
\begin{bmatrix}
\alpha_1 & a_{1,2} & \cdots & a_{1,r+1} \\
a_{1,2} & \alpha_2 & \cdots & a_{2,r+1} \\
\vdots & \vdots & \ddots & \vdots \\
a_{1,r+1} & a_{2,r+1} & \cdots & b
\end{bmatrix}_{(r+1)\times(r+1)}.
$
In other words, $H_1$ is the principal 
$(r+1)\times (r+1)$ submatrix of $H$ consisting of the first $r+1$ rows and columns,
except the entry $\alpha_{r+1}$ is 
replaced by $b$. 
Then $H_1\in\mathcal G_{r+1}$. 
If $\alpha_{r+1}=b$, then $H_1$, 
being a principal submatrix of $H$, is positive semidefinite. We see that
$\left(\mathcal{J}_H : x_{r+1}^{\alpha_{r+1}-b}\right)
= \left\langle \mathcal{J}_{H_1},
~x_l^{\alpha_l-b} : r+2\leq l \leq n \right\rangle$,
where $\mathcal{J}_{H_1}\subseteq R_{r+1}
= \mathbb K[x_1,\ldots ,x_{r+1}]$. Thus
\begin{align}\label{e7}
\dim_{\mathbb K}\left(\frac{R_n}{\left(\mathcal{J}_H : x_{r+1}^{\alpha_{r+1}-b}\right)} \right)
= \dim_{\mathbb K}\left(\frac{R_{r+1}}
{\mathcal{J}_{H_1}}\right) ~
\left(\prod_{l=r+2}^n(\alpha_l-b)\right) .
\end{align}
Let $H_2$ be the $(n-1)\times(n-1)$ submatrix of $H$ obtained on deleting $(r+1)$th row and $(r+1)$th column. As $H_2\in\mathcal G_{n-1}$ is positive semidefinite,
 the monomial ideal $\mathcal{J}_{H_2}
 \subseteq \mathbb K\left[x_1,\ldots ,\hat{x}_{r+1},\ldots ,x_n\right]=R_{n-1}$ satisfies 
 $\dim_{\mathbb K}\left(\frac{R_{n-1}}
 {\mathcal{J}_{H_2}}\right) \geq \det H_2$, 
 by induction assumption. 
 Also $\left\langle \mathcal{J}_H,~
 x_{r+1}^{\alpha_{r+1}-b} \right\rangle
 = \left\langle \mathcal{J}_{H_2},~
 x_{r+1}^{\alpha_{r+1}-b} \right\rangle$. 
 Thus 
\begin{align}\label{e8}
\dim_{\mathbb K}\left(\frac{R_n}
{\left\langle \mathcal{J}_H,
~x_{r+1}^{\alpha_{r+1}-b} \right\rangle}\right) = (\alpha_{r+1}-b) ~ \dim_{\mathbb K}\left(\frac{R_{n-1}}{\mathcal{J}_{H_2}}\right).
\end{align}
From (\ref{e6}), (\ref{e7}) and (\ref{e8}), 
we have
\begin{align}\label{e9}
\dim_{\mathbb K}\left(\frac{R_n}
{\mathcal{J}_H}\right) =
\left(\prod_{l=r+2}^n(\alpha_l-b)\right)~
\dim_{\mathbb K}\left(\frac{R_{r+1}}
{\mathcal{J}_{H_1}}\right)
+ (\alpha_{r+1}-b) ~
\dim_{\mathbb K}\left(\frac{R_{n-1}}
{\mathcal{J}_{H_2}}\right).
\end{align}
As determinant is linear on columns, writing 
$\alpha_{r+1}=(\alpha_{r+1}-b)+b$ in $H$, we have
\begin{align}\label{e27}
\det H=(\alpha_{r+1}-b)\det H_2+\det T,
\end{align}
where $T$ is the matrix $H$, 
 except $\alpha_{r+1}$ is replaced with $b$.
  On applying elementary column and row operations, $C_{r+2}-C_{r+1},R_{r+2}-R_{r+1},\ldots, C_n-C_{r+1},R_n-R_{r+1}$ on $T$, it reduces to the matrix
\[
T'=
\begin{bmatrix}
\alpha_1 & \cdots & a_{1,r} & a_{1,r+1} & a_{1,r+2}-a_{1,r+1} & \cdots & a_{1,n}-a_{1,r+1} \\
\vdots & \ddots & \vdots & \vdots & \vdots & \ddots & \vdots \\
a_{1,r} & \cdots & \alpha_r & a_{r,r+1} & a_{r,r+2}-a_{r,r+1} & \cdots & a_{r,n}-a_{r,r+1} \\
a_{1,r+1} & \cdots & a_{r,r+1} & b & 0 & \cdots & 0 \\
a_{1,r+2}-a_{1,r+1} & \cdots & a_{r,r+2}-a_{r,r+1} & 0 & \alpha_{r+2}-b & \cdots & a_{r+2,n}-b \\
\vdots & \ddots & \vdots & \vdots & \vdots & \ddots & \vdots \\
a_{1,n}-a_{1,r+1} & \cdots & a_{r,n}-a_{r,r+1} & 0 & a_{r+2,n}-b & \cdots & \alpha_n-b  
\end{bmatrix}_{n\times n}
\]
Let $\varepsilon_{i,j}$ be the $n\times n$ matrix with $1$ at $(i,j)$th place and zero elsewhere. 
Then $P=I_n - \sum_{j=r+2}^n ~\varepsilon_{r+1,j}
= I_n - \left(\varepsilon_{r+1,r+2}+ \ldots +
\varepsilon_{r+1,n} \right)$ has determinant 
$\det P=1$ and $P^t T P = T'$.
Thus $\det T=\det T'$. Now we consider two cases.

\noindent
{\bf Case I :}
 $\det T \leq 0$. Then from (\ref{e27}), 
 $\det H \leq (\alpha_{r+1}-b) \det H_2$.
Thus by induction assumption and (\ref{e9}), we get
\[
\det H ~\leq ~ (\alpha_{r+1}-b)~ \dim_{\mathbb K}\left(\frac{R_{n-1}}{\mathcal{J}_{H_2}}\right) ~
\leq ~ \dim_{\mathbb K}\left(\frac{R_n}{\mathcal{J}_H}\right).
\]
{\bf Case II :}
$\det T>0$. If $\alpha_{r+1}=b$, then $H=T$ 
is positive definite.
Otherwise, $H=T+S$, where 
$S=(\alpha_{r+1}-b)~\varepsilon_{r+1,r+1}$.
Clearly, $\lambda_1(S)=\cdots =\lambda_{n-1}(S)=0$ and $\lambda_n(S)=\alpha_{r+1}-b$.
Since $H$ is positive semidefinite, $0\leq\lambda_1(H)\leq\lambda_2(H)\leq\cdots\leq\lambda_n(H)$.
Taking $i=j=1$ in the Courant-Weyl inequalities with $H=T+S$, we obtain $\lambda_1(H)\leq\lambda_2(T)+\lambda_{n-1}(S)=\lambda_2(T)$.
Thus $0\leq\lambda_2(T)\leq\ldots\leq\lambda_n(T)$. As $\det(T)=\prod_{i=1}^n\lambda_i(T)>0$, $T$ must be positive definite. Hence $T'=P^t T P$ is also positive definite. Thus by Fischer's inequality,
\[
\det T=\det T'\leq\det(H_1)\det(C),
\]
where $C=
\begin{bmatrix}
\alpha_{r+2}-b & \cdots & a_{r+2,n}-b \\
\vdots & \ddots & \vdots \\
a_{r+2,n}-b & \cdots & \alpha_n-b
\end{bmatrix}.
$
The matrix $C$, being a principal submatrix of
$T'$, is also positive definite. Thus by Hadamard's theorem, $\det C ~\leq ~\prod_{l=r+2}^n(\alpha_l-b)$. Hence,
\begin{align}\label{e10}
\det(T) ~\leq ~ \left(\prod_{l=r+2}^n(\alpha_l-b)\right)\det(H_1).
\end{align}
From (\ref{e27}) and (\ref{e10}), 
\begin{align*}
\det H ~\leq ~\left(\prod_{l=r+2}^n(\alpha_l-b)\right)
~ \det H_1+(\alpha_{r+1}-b)~\det H_2.
\end{align*}
Now by (\ref{e9}) and induction assumption, we have
\begin{align*}
\dim_{\mathbb K}\left(\frac{R_n}{\mathcal{J}_H}\right)~
\geq~ \left(\prod_{l=r+2}^n(\alpha_l-b)\right)~
\det H_1 + (\alpha_{r-1}-b)~\det H_2 
~\geq ~ \det H.
\end{align*}
\hfill $\square$
\end{proof}
\begin{corollary}\label{MG}
Let $G$ be a multigraph on $V=\{0,1,\ldots ,n\}$. Then 
\[
\dim_{\mathbb K}\left(\frac{R_n}{\mathcal{M}_G^{(1)}}\right)
~ \geq ~ \det\widetilde Q_G.
\]
\end{corollary}
\begin{proof}
Take $H=\widetilde Q_G$ in Theorem \ref{MT}.
\end{proof}

{\bf Acknowledgments:} The third author is thankful
to CSIR, Government of India for financial support.

\end{document}